\documentclass[12pt]{article}
\usepackage{latexsym}
\usepackage{amsmath,amsfonts,amssymb,mathrsfs,graphicx}
\usepackage{amsthm}
\usepackage{verbatim}
\usepackage{setspace}

\title{A capacity approach to box and packing dimensions of projections of sets and exceptional directions}
\author{K.J. Falconer\\
\small{{\it Mathematical Institute,  
University of St~Andrews, North Haugh, St~Andrews,}} \\
\small{{\it Fife, KY16~9SS, Scotland }}\\
\small{{\sf kjf@st-andrews.ac.uk }}} 
\date{}

\def\bbbr{\mathbb{R}}
\def\rn{\mathbb{R}^n}

\newcommand\ubd{\overline{\mbox{\rm dim}}_{\rm B}} 
\newcommand\lbd{\underline{\mbox{\rm dim}}_{\rm B}} 
\newcommand\bdd{\mbox{\rm dim}_{\rm B}}
\newcommand\pkd{\mbox{\rm dim}_{\rm P}} 
\newcommand\hdd{\mbox{\rm dim}_{\rm H}} 

\newcommand{\be}{\begin{equation}} 
\newcommand{\ee}{\end{equation}} 

\newcommand\pve{\pi_V E}

\allowdisplaybreaks

 \newtheorem{theo}{Theorem}[section]
 \newtheorem{cor}[theo]{Corollary}
 \newtheorem{lem}[theo]{Lemma}
 \newtheorem{prop}[theo]{Proposition}

\begin{document}
\maketitle

\begin{abstract}
\noindent Dimension profiles were introduced in \cite{FH2, How} to give a formula for the box-counting and packing dimensions of the orthogonal projections of a set $E \subset \rn$ onto almost all $m$-dimensional subspaces. However, these definitions of dimension profiles are indirect and are hard to work with. Here we firstly give alternative definitions of dimension profiles in terms of capacities of $E$  with respect to certain kernels, which lead to the box-counting and packing dimensions of projections fairly easily, including estimates on the size of the exceptional sets of subspaces where the dimension of projection is smaller than the typical value. Secondly, we argue that with this approach projection results for different types of dimension may be thought of in a unified way. Thirdly, we use a Fourier transform method to obtain further inequalities on the size of the exceptional subspaces.  
 \end{abstract}

\section{Introduction and main results}
\setcounter{equation}{0}
\setcounter{theo}{0}

\subsection{Introduction}

The relationship between the Hausdorff dimension of a set $E \subset \rn$ and of its orthogonal projections $\pi_V(E)$ onto  subspaces $V\in G(n,m)$, where $G(n,m)$ is the Grassmanian of $m$-dimensional subspaces of $\rn$ and $\pi_V: \rn \to V$ denotes orthogonal projection, has been studied since the foundational work of Marstrand \cite{Mar} and Mattila \cite{Mat4}. They showed that for Borel $E \subset \rn$ 
\be\label{marmat}
\hdd \pi_V(E) \ = \ \min\{\hdd E, m\}  
\ee
for almost all $m$-dimensional subspaces $V$ (with respect to the natural invariant probability measure $\gamma_{n,m}$ on $G(n,m)$) where $\hdd$ denotes Hausdorff dimension. Kaufman \cite{Kau,KM} used capacities to prove and extend these results and this has become the standard approach for  such problems.
There are many generalisations, specialisations and consequences of these projection results, see \cite{FFJ,Mat5} for recent surveys.

It is natural to seek analogous projection results for other notions of dimension. However, examples show that the direct analogue of \eqref{marmat} is not valid for lower or upper box-counting (Minkowski) dimensions or packing dimension, though there are non-trivial lower bounds on the dimensions of the projections, see \cite{FH,FM,Jar}.  It was shown in \cite{FH2,How} that the box-counting and packing  dimensions of $\pi_V(E)$ for a Borel set $E$ are constant for almost all $V \in G(n,m)$ but this constant value, termed a  `dimension profile' of $E$, had a very indirect definition in terms of the supremum of dimension profiles of measures supported by $E$ which in turn are given by critical parameters for certain almost sure pointwise limits  \cite{FH2}. A later approach in \cite{How} defines box-counting dimension profiles in terms of weighted packings subject to constraints.  

I was never very happy with these definitions, which are artificial, indirect and awkward to use. To make the concept more attractive and useful, this paper presents an alternative and more natural way of defining box-counting and packing dimension profiles in terms of capacities with respect to certain kernels. Then using simple properties of equilibrium measures we can find 
the `typical' box or packing dimensions of $\pi_V(E)$, that is those that are realised for almost all $V \in G(n,m)$, as a dimension profile of $E$. With little more effort, we can also obtain some upper bounds for the dimension of the exceptional $V \in G(n,m)$ where the projection dimension is smaller than this typical value. Then, using Fourier transform methods, we will obtain new estimates on the dimension of the exceptional sets of $V \in G(n,m)$ for box and packing dimensions when, roughly speaking, the dimension of $E$ is greater than $m$.

Thus in \eqref{dimpro} we will define the  $s$-box dimension profile of $E\subset \mathbb{R}^n$ for $s> 0$ as
$$\bdd^s E\  =\   \lim_{r\to 0} \frac{\log  C_r^s(E)}{-\log r}, $$
 where $C_r^s(E)$ is the capacity of $E$ with respect to the continuous kernel \eqref{ker} (more precisely taking  lower and upper limits  will give the lower and upper dimension profiles). We will show in Section \ref{sec2.2}
 that if $s\geq n$ then $\bdd^s E$ is just the usual box-counting dimension of $E$.  On the other hand, in Section \ref{sec3.1}, we show that if $1 \leq m \leq n-1$, then $\bdd^m E$ equals the box-counting dimension of  $\pi_V(E)$ for almost all $V \in G(n,m)$. In this way, the dimension profile $\bdd^s E$ may be thought of as the dimension of $E$ when regarded from an $s$-dimensional viewpoint. Analogously,  $\hdd^s E= \min\{\hdd E, s\}$ might be interpreted as  the Hausdorff dimension profile for the Hausdorff dimension result \eqref{marmat}. By defining packing dimension profiles in terms of upper box dimension profiles in Section \ref{Secpac} we obtain similar results for the packing dimension of projections. In Section \ref{ineqs} we consider inequalities satisfied by the dimension profiles which help give a feel for the results.

Since their conception, dimension profiles have also become a key tool for investigating the packing and box dimensions of the images of sets under random processes, see for example \cite{Fa2,SX,Xi}.

\subsection{Main results on projections and exceptional directions}\label{mainres}

Given the definitions of the dimension profiles which will be formally defined in \eqref{dimpro}, the basic projection results are easily stated. Essentially, the $m$-dimension profiles of $E$ give the dimension of the projections of $E$ onto almost all $m$-dimensional subspaces. Theorem \ref{mainA} is the basic result on dimension of projections, and Theorems \ref{mainB} and \ref{mainC} concern the dimensions of the set of $V\in G(n,m)$ for which the dimensions of the projections onto $V$ are exceptionally small.
We include the well-known Hausdorff dimension projection results for comparison which are directly analogous to the conclusions for box and packing dimension if we define 
$$
\hdd^s E:= \min\{s, \hdd E\}
$$
to be the {\it Hausdorff dimension profile} of $E$.

\begin{theo}\label{mainA}
Let $E\subset \mathbb{R}^n$ be a non-empty Borel set (assumed to be bounded in $(ii)$ and $(iii)$). Then for  all $V\in G(n,m)$,
\vspace{-0.5cm} 
{\setstretch{1.4}
\begin{align*}
(i)\quad & \hdd \pi_V E \  \leq \  \hdd^m E\ \equiv \  \min\{m, \hdd E\},\\
(ii)\quad & \lbd \pi_V E \  \leq \   \lbd^m E,\\
(iii)\quad &\ubd \pi_V E \  \leq \   \ubd^m E,\\
(iv)\quad &\pkd \pi_V E \  \leq \   \pkd^m E,
\end{align*}
}
with equality in all of the above for $\gamma_{n,m}$-almost all $V\in G(n,m)$.
\end{theo}

Part (i) of Theorem \ref{mainA} goes back to Marstrand \cite{Mar} and Mattila \cite{Mat4}, and parts (ii)-(iv) were obtained in \cite{FH2,How} but starting with the original cumbersome definitions of the box and packing dimension profiles. After relating capacities and box-counting numbers in Section \ref{seccaps}, parts (ii) and (iii) will follow easily, and (iv) will come from the relationship between packing and box-dimension profiles discussed in Section \ref{Secpac}.

To put these estimates into context,  $\pkd^m E$, etc. cannot be too small compared with  $\pkd E$. Indeed
\be\label{projin}
 \frac{\pkd E}{1 +(1/m -1/n)\pkd E}\ \leq\ \pkd^m E \ \leq \min\{\pkd E, m\};
\ee
these bounds are sharp and there are identical inequalities for $\lbd$ and $\ubd$, see  Section \ref{ineqs} and \cite{FH}. Thus the almost sure dimensions of the projections are also constrained by these bounds.

Whilst equality holds in  Theorem \ref{mainA} for $\gamma_{n,m}$-almost all $V \in G(n,m)$, dimension profiles can provide further information on the size of the set of $V$ for which the box dimensions of the projections $ \pi_V E$ are exceptionally small.
Note that $ G(n,m)$ is a manifold of dimension  $m(n-m)$ so $\hdd G(n,m) =   m(n-m)$ and it is convenient to express our estimates relative to this dimension.
Theorem \ref{mainB} gives estimates for the Hausdorff dimensions of the exceptional sets in terms of $\lbd^s E$, etc. when $0\leq s\leq m$ and Theorem \ref{mainC} gives estimates when $m \leq s\leq n$.

\begin{theo}\label{mainB}
Let $E\subset \mathbb{R}^n$ be a non-empty Borel set (assumed to be bounded in $(ii)$ and $(iii)$) and let $0\leq s\leq m$. Then 
\vspace{-0.5cm} 
{\setstretch{1.5}
\begin{align*}\label{dbeqa}
(i)\quad & \hdd \{V \in  G(n,m): \hdd \pi_V E \,  < \,    \hdd^s E\}
\   \leq\  m(n-m) - (m-s),\\
(ii)\quad & \hdd \{V \in  G(n,m): \lbd \pi_V E \,  < \,   \lbd^s E\}
\   \leq\  m(n-m) - (m-s),\\
(iii)\quad & \hdd \{V \in  G(n,m): \ubd \pi_V E \,  < \,   \ubd^s E\}
\   \leq\  m(n-m) - (m-s),\\
(iv)\quad &\hdd \{V \in  G(n,m) : \pkd \pi_V E \,  < \,   \pkd^s E\}
\    \leq \  m(n-m) - (m-s).
\end{align*}
}
\end{theo}
Noting that $\lbd^s E$, etc. increases with $s$, these bounds for the Hausdorff dimension of the exceptional sets of $V$ decrease as  $s$  decreases.

Using capacity ideas, all parts of Theorem \ref{mainB} may be derived using minor modifications to the proofs for Theorem \ref{mainA}. Part (i) was first obtained by Kaufman \cite{Kau} when he introduced the potential theoretic approach for the Hausdorff dimension of projections. Parts (ii)-(iv) were established in \cite{FH2,How} using the earlier definitions of dimension profiles, but the proofs here using capacities are rather simpler. 

The conclusions of Theorem \label{mainB} can be strengthened slightly: in all parts (ii)-(iv) the ${m(n-m) - (m-s)}$-dimensional Hausdorff measure of the sets on the left must be 0, see the note at the end of Section \ref{sec3.1}. However for consistency it seems more natural to state the theorem in terms of Hausdorff dimension.

The spirit of the next theorem is that if the dimension of $E$ is significantly larger than that of the typical projection given by Theorem   \ref{mainA} then the exceptional set of $V$ will be small.

\begin{theo}\label{mainC}
Let $E\subset \mathbb{R}^n$ be a non-empty Borel set (assumed to be bounded in $(ii)$ and $(iii)$) and let $0\leq \gamma \leq n-m$. Then 
\vspace{-0.5cm} 
{\setstretch{1.5}
\begin{align*}
(i)\quad & \hdd \{V \in  G(n,m): \hdd \pi_V E \,  < \,   \hdd^{m+\gamma} E -\gamma\}
\   \leq\  m(n-m) -\gamma,\\
(ii)\quad & \hdd \{V \in  G(n,m) : \lbd \pi_V E \,  < \,   \lbd^{m+\gamma}E - \gamma \}
\ \leq\  m(n-m)  -\gamma,\\
(iii)\quad & \hdd \{V \in  G(n,m) : \ubd \pi_V E \,  < \,   \ubd^{m+\gamma}E - \gamma \}
\ \leq\  m(n-m) -\gamma,\\
(iv)\quad &\hdd \{V \in  G(n,m) : \pkd \pi_V E \,  < \,   \pkd^{m+\gamma}E - \gamma\}
\    \leq \ m(n-m) -\gamma.
\end{align*}
}
\end{theo}
These estimates are expressed in terms of $ \lbd^{m+\gamma}E - \gamma$, etc. since $\gamma$ cannot easily be isolated from such expressions (except in case (i)). It follows from inequality \eqref{ineq3} derived in Section \ref{ineqs} that  $\lbd^{m+\gamma}E - \gamma$, etc. decrease continuously as $\gamma$ increases, equalling the typical projection dimension  $\lbd^m E$ when $\gamma= 0$, and $\lbd E-(n-m)$ when $\gamma= n-m$. (Of course the estimates become trivial unless  $0\leq \lbd^{m+\gamma}E - \gamma\leq m$.)

Note that the Hausdorff dimension parts (i) of Theorems \ref{mainB} and  \ref{mainC} can be rearranged to the more familiar form
\begin{align*}
 \hdd \{V \in & G(n,m): \hdd \pi_V E \,  < s\}\\
&\leq\ 
 \left\{
 \begin{array}{ll}
m(n-m) - (m-s) & \quad (0\leq s\leq m)\\
m(n-m) - (\hdd E-s) & \quad (\hdd E-(n-m) \leq s\leq \hdd E)
\end{array}
\right. .
\end{align*}

Case $(i)$ of Theorem \ref{mainC} was established in \cite{Fa1}, using Fourier transforms, see also \cite{Mat2}. We will use Fourier methods to obtain the box  dimension cases $(ii)$-$(iii)$, from which we will deduce $(iv)$ . 

We remark that other recent delicate estimates have been given by \cite{Obe} using Radon transform estimates and by \cite{Bou,He} using ideas from additive combinatorics.

\section{Capacities and box-counting dimensions}\label{seccaps}
\setcounter{equation}{0}
\setcounter{theo}{0}

Throughout this section we will consider projections of a  Borel set  $E\subset \rn$ which we will take to be non-empty and bounded to ensure that its box dimensions are defined. Moreover, since the lower and upper box dimensions and the capacities of a set equal those of its closure, it is enough to prove our results under the assumption that $E$ is non-empty and compact.

\subsection{Capacity, energy and dimension profiles}
Potential kernels of the form $\phi(x) = |x|^{-s}$ are widely used in Hausdorff dimension arguments, see for example \cite{Kau,KM,Mat,Mat2}. For box-counting dimensions, another class of kernels turns out to be useful.
Let $s>0$ and $r>0$ and define the {\it potential kernels} 
\be\label{ker}
 \phi_r^s(x)= \min\Big\{ 1, \Big(\frac{r}{|x|}\Big)^s\Big\} \qquad (x\in \rn),
\ee
originally introduced in \cite{FH,FM}.
Let $E\subset \rn$ be non-empty and compact and let   ${\mathcal M}(E)$ denote the set of Borel probability measures supported by $E$. The {\em energy} of  $\mu  \in {\mathcal M}(E)$ with respect to  $\phi_r^s$ is defined by
$$\int\int \phi_r^s(x-y)d\mu(x)d\mu(y),$$
and the {\em potential}  of $\mu$  at $x\in \rn$ by 
$$ \int \phi_r^s(x-y)d\mu (y).$$
The {\em capacity} $C_r^s(E) $ of $E$  is the reciprocal of the minimum energy achieved by probability measures on $E$, that is 
\be\label{minen}
\frac{1}{C_r^s(E)}\  = \ \inf_{\mu \in {\mathcal M}(E)}\int\int \phi_r^s(x-y)d\mu(x)d\mu(y);
\ee
 since  our kernels $\phi_r^s$ are continuous and $E$ is  compact, $0<C_r^s(E)<\infty$. For a general bounded set the {\it capacity} is defined to be that of its closure.

The following energy-minimising property is standard in potential theory, but it is key for our development, so we give the short proof which is particularly simple for continuous kernels.

\begin{lem}\label{equilib}
Let $E\subset \rn$ be non-empty and compact and $s>0$ and $r>0$.  Then the infimum in \eqref{minen} is attained by a measure $\mu_0  \in {\mathcal M}(E)$. Moreover 
\be\label{pot}
 \int \phi_r^s(x-y)d\mu_0  (y)\ \geq\  \frac{1}{C_r^s(E)}
\ee
for all $x\in E$, with equality for $\mu_0 $-almost all $x \in E$.
\end{lem}
\begin{proof}
Let $\mu_k \in {\mathcal M}(E)$ be such that $\int\int \phi_r^s(x-y)d\mu_k(x)d\mu_k(y) \to \gamma := 1/C_r^s(E)$. Then $\mu_k$ has a  subsequence that is weakly convergent to some $\mu_0  \in {\mathcal M}(E)$. Since $\phi_r^s(x-y)$ is continuous the infimum is attained.

Suppose that $\int \phi_r^s(z-y)d\mu_0(y)\leq \gamma - \epsilon$ for some $z\in E$ and $\epsilon >0$. Let $\delta_z$ be the unit point mass at $z$ and for $0<\lambda<1$ let $\mu_\lambda =\lambda \delta_z + (1-\lambda)\mu_0  
\in {\mathcal M}(E)$. Then
\begin{eqnarray*}
\int\int \phi_r^s(x-y)d\mu_\lambda(x)d\mu_\lambda(y)
&=& \lambda^2 \phi_r^s(z-z) +2\lambda(1-\lambda)\int \phi_r^s(z-y)d\mu_0 (y)\\
&&+\ (1-\lambda)^2 \int\int \phi_r^s(x-y)d\mu_0 (x)d\mu_0 (y)\\
&\leq& \lambda^2 +2\lambda(1-\lambda)(\gamma - \epsilon) +(1-\lambda)^2\gamma\\
&=& \gamma - 2\lambda\epsilon + O( \lambda^2),
\end{eqnarray*}
which contradicts that $\mu_0 $ minimises the energy integral on taking $\lambda$ sufficiently small.
Thus  inequality \eqref{pot} is satisfied for all $x\in E$, and equality for $\mu_0 $-almost all $x$ is immediate from \eqref{minen}. 
\end{proof}

For $s> 0$ we define the {\it lower} and {\it upper} $s$-{\it box dimension profiles} of $E\subseteq \mathbb{R}^n$ in terms of capacities:
\be\label{dimpro}
\lbd^s E\  =\   \varliminf_{r\to 0} \frac{\log  C_r^s(E)}{-\log r} \quad \mbox{and}\quad  \ubd^s E\   =\   \varlimsup_{r\to 0} \frac{\log  C_r^s(E)}{-\log r}.
\ee
Note that $\lbd^s E = \lbd^s \overline{E}$ and  $\ubd^s E = \ubd^s \overline{E}$   where $\overline{E}$ denotes the closure of $E$.

\subsection{Capacities and box-counting numbers}\label{sec2.2}

For a non-empty compact  $E\subset \rn$,  let $N_r(E)$ be the minimum number of sets of diameter $r$ that can cover $E$. Recall that the {\em lower} and {\em upper box-counting dimensions} or {\em box dimensions} of $E$ are defined by 
\be\label{boxdims}
\lbd E\ =\ \varliminf_{r\to 0} \frac{\log  N_r(E)}{-\log r} 
\quad \mbox{ and }\quad  \ubd E\ = \ \varlimsup_{r\to 0} \frac{\log  N_r(E)}{-\log r}, 
\ee
with the {\em box-counting dimension} given by the common value if the limit exists, see for example \cite{Fa} for a discussion of box dimensions and equivalent definitions; in particular the box dimensions of a set equal those of its closure.

In this section we prove Corollary  \ref{capcor}, that provided that $s\geq n$ the capacity $C_r^s(E)$ and the covering number $N_r(E)$ are comparable. This is not necessarily the case if $0\leq s<n$ and  it is this disparity that gives the formulae for the box dimensions of projections. The next two lemmas obtain lower and upper bounds for $N_r(E)$ in terms of energies or potentials.

\begin{lem}\label{genbound}
Let $E\subset \rn$ be non-empty and compact and let $r>0$. 
Suppose that there is a measure  $\mu\in {\mathcal M}(E)$ such that for some $\gamma>0$
\be\label{est0}
(\mu \times \mu)\big\{(x,y): |x-y| \leq r\big\}\ \leq\ \gamma.
\ee
Then 
\begin{equation}\label{concl1}
N_{r}(E) \geq \frac{c_n}{\gamma},
\end{equation}
where  $c_n$ depends only on $n$. In particular \eqref{concl1} holds if, for some $s>0$,
\be\label{est0a}
\int\int \phi_r^s(x-y)d\mu(x)d\mu(y)\ \leq\ \gamma.
\ee
\end{lem}

\begin{proof}
Let ${\cal C}(E)$ be the set of closed coordinate mesh cubes of diameter $r$ (i.e. cubes of the form $\prod_{i=1}^n [m_i rn^{-1/2},(m_i +1) rn^{-1/2}]$ where the $m_i$ are integers)
 that intersect $E$;  suppose that there are  $N'_{r}(E)$ such cubes.
Using  Cauchy's inequality,
\begin{eqnarray*}
1\ =\ \mu(E)^2    &\leq & \bigg(\sum_{C\in {\cal C}(E)}\mu(C)\bigg)^2\\
& \leq &  N'_{r}(E) \sum_{C\in {\cal C}(E)}\mu(C) ^2\\
& = &  N'_{r}(E) \sum_{C\in {\cal C}(E)}(\mu \times \mu)\big\{(x,y)\in C\times C\big\}\\
& \leq &  N'_{r}(E)\, (\mu \times \mu) \big\{(x,y): |x-y| \leq r\big\}\\
& \leq &  N'_{r}(E)\,  \gamma\\
& \leq &  (3 \sqrt{n})^{n}\, N_{r}(E)\,  \gamma,
\end{eqnarray*}
noting that a set of diameter $r$ can intersect at most $(3 \sqrt{n})^{n}$ of the cubes of ${\cal C}(E)$.

Finally since $1_{B(0,r)}(x-y)\leq \phi_r^s(x-y)$,  inequality \eqref{est0a}  implies \eqref{est0}.
\end{proof}

\begin{lem}\label{potbound}
Let $E\subset \rn$ be non-empty and compact and let $s>0$ and $r>0$. 
Suppose that $E$ supports a measure  $\mu\in {\mathcal M}(E)$ such that for some $\gamma>0$
\be\label{est2}
\int \phi_r^s(x-y)d{\mu} (y)\ \geq \ \gamma \qquad \mbox{ for all } x\in E.
\ee
Then
\be\label{est2c}
N_r(E)\  \leq \  
\left\{
\begin{array}{ll}
{\displaystyle \frac{c_{n,n}\lceil\log_2 ({\rm diam}E / r)+1\rceil}{\gamma } }& \mbox{ if } s=n    \\
{\displaystyle  \frac{c_{n,s}}{\gamma}}&  \mbox{ if } s>n   
\end{array}
\right.  ,
\ee
where $c_{n,s}$ depends only on $n$ and $s$.
\end{lem}

\begin{proof}
 Write $M = {\rm diam}E$. For all $x\in E$,
\begin{eqnarray*}
 \int \phi_r^s(x-y)d{\mu} (y)
&\leq& \mu(B(x,r)) +\sum_{k=0}^{\lceil\log_2 (M/r)-1\rceil} \int _{B(x,2^{k+1}r)\setminus  B(x,2^{k}r)}2^{-ks} d\mu(y)\\
&\leq& \mu(B(x,r)) +\sum_{k=0}^{\lceil\log_2 (M/r)-1\rceil} 2^{-ks} \mu(B(x,2^{k+1}r)) \\
&\leq& 2^s\sum_{k=0}^{\lceil\log_2 (M/r)\rceil} 2^{-ks} \mu(B(x,2^{k}r)).
\end{eqnarray*}
Let $B(x_i, r),\  i = 1,\ldots, N'_r(E)$, be a maximal collection of disjoint balls of radii $r$ with $x_i \in E$, where  here $N'_r(E)$ denotes this maximum number.
From \eqref{est2}, for each $i$,
$$ \gamma\ \leq\   \int \phi_r^s(x_i-y)d{\mu} (y)\ \leq2^s\ \sum_{k=0}^{\lceil\log_2 (M/r)\rceil} 2^{-ks} \mu(B(x_i,2^{k}r)).$$
Summing over the $x_i$,
$$N'_r(E)\gamma\  \leq   \sum_{k=0}^{\lceil\log_2 (M/r)\rceil} 2^{s(1-k)} \sum_{i=1}^{N'_r(E)}  \mu(B(x_i,2^{k}r)),$$
so, for some $k$ with $0\leq k \leq\lceil \log_2 (M/r)\rceil$,
\be\label{est3}
2^{s(1-k)}\sum_{i=1}^{N'_r(E)}  \mu(B(x_i,2^{k}r))  \ \geq\
\left\{
\begin{array}{ll}
 N'_r(E)\gamma\big/  \lceil\log_2 (M/r)+1\rceil   & \mbox{ if } s=n    \\
 N'_r(E)\gamma\, 2^{k(n-s)}  (1-2^{n-s}) &  \mbox{ if } s>n   
\end{array}
\right. ,
\ee
the case of $s>n$ coming from comparison with a geometric series. For all $x\in E$ a volume estimate using the disjoint balls $B(x_i,r)$ shows that at most $ (2^k +1)^n\leq 2^{(k+1)n}$ of the $x_i$ lie in $B(x,2^{k}r)$. Consequently each  $x$ belongs to at most $2^{(k+1)n}$ of the $B(x_i,2^{k}r)$. Thus 
\be\label{est4}
\sum_{i=1}^{N'_r(E)}  \mu(B(x_i,2^{k}r))\ \leq\ 2^{(k+1)n}\mu(E)\ 
=\ 2^{n+s}2^{-s(1-k)}2^{k(n-s)}\ \leq \ 2^{n+s} 2^{-s(1-k)},
\ee
using that $s\geq n$.
Inequality  \eqref{est2c} now  follows from \eqref{est3},  \eqref{est4} and  that $N_r(E)\leq a_n N'_r(E)$ where $a_n$ is the minimum number of balls  in $\mathbb{R}^n$ of diameter $1$ that can cover a ball of radius 1.
\end{proof}

The comparability of box-counting numbers and capacities for $s\geq n$  now follows on combining the previous two lemmas.

\begin{cor}\label{capcor}
Let $E\subset \rn$ be non-empty and bounded and let $r>0$. Then
\be\label{capineq}
c_n C^s_r(E)\  \leq\ N_r(E)\  \leq \  
\left\{
\begin{array}{ll}
{\displaystyle c_{n,n}\lceil\log_2 ({\rm diam}E / r)\rceil\  C^s_r(E) }& \mbox{ if } s=n    \\
{\displaystyle  c_{n,s}\ C^s_r(E)}&  \mbox{ if } s>n   
\end{array}
\right.  .
\ee
\end{cor}

\begin{proof}
By Lemma \ref{equilib}  we may find $\mu\in {\mathcal M}(E)$ satisfying  \eqref{pot}, so the conclusion when $E$ is compact follows immediately from Lemmas \ref{genbound} and \ref{potbound}. For general bounded $E$, since $C^s_r(E)=C^s_r(\overline{E})$ and $N_r(E)=N_r(\overline{E})$, where $\overline{E}$ is the closure of $E$, the conclusion transfers directly to all non-empty bounded $E$.
\end{proof}

Equality of the box dimensions and the dimension profiles for $s\geq n$ is immediate from  Corollary \ref{capcor}. 

\begin{cor}\label{sgeqn}
Let $E\subset \mathbb{R}^n$ be non-empty and bounded. If $s\geq n$ then 
$$\lbd^s E\   =\   \lbd E  \quad \mbox{and}\quad  \ubd^s E\   = \  \ubd E.$$
\end{cor}
\begin{proof}
This follows from \eqref{capineq}   and the definitions of box dimensions \eqref{boxdims} and of dimension profiles \eqref{dimpro}.
\end{proof}

\section{Proofs of the projection results}
\setcounter{equation}{0}
\setcounter{theo}{0}
In this section we prove parts (ii) and (iii) of the theorems stated in Section \ref{mainres} concerning the lower and upper box dimensions of projections. Parts (iv) on packing dimensions will follow from the relationships between  packing dimension and upper box dimension and their dimension profiles which will be discussed in Section \ref{Secpac}.

\subsection{Proofs of Theorems \ref{mainA} and \ref{mainB} parts (ii) and (iii)}\label{sec3.1}

The upper bound for the dimensions of projections onto subspaces is an easy consequence of the way that the kernels behave under projections together with the relationship between box dimensions and capacities from Lemma \ref{potbound}.
\medskip

\noindent{\it Proof of Theorem \ref{mainA} (ii) and (iii) (inequalities).}
It is enough to obtain the upper bound when $E$ is compact.
Let $ V\in G(n,m)$ and  $r>0$. Since $\pi_V$ does not increase distances, 
\begin{align*}
 \phi_r^m(\pi_V(x)-\pi_V(y)) & =  \min\Big\{ 1, \Big(\frac{r}{|\pi_V(x)-\pi_V(y)|}\Big)^m\Big\}\\
& \geq\  \min\Big\{ 1, \Big(\frac{r}{|x-y|}\Big)^m\Big\}\
 =\  \phi_{r}^{m}(x-y)\qquad\qquad (x, y \in E).
\end{align*}
For each $r>0$ we may, by Lemma \ref{equilib}, find a measure $\mu  \in {\mathcal M}(E)$ such that for all $x\in E$
$$\frac{1}{C^{m}_{r }(E)} \leq \int \phi_{r}^{m}(x-y)d\mu  (y)\ 
 \leq \ \int \phi_r^m(\pi_V(x)-\pi_V(y))d\mu (y)\ 
=\ \int \phi_r^m(\pi_V(x)-w)d\mu_V (w),$$
where  $\mu_V  \in {\mathcal M}(\pve)$ is  the image of the measure $\mu$ under $\pi_V$, defined by $\int g(w) d\mu_V (w) =  \int g(\pi_Vx) d\mu(x)$ for continuous $g:V\to\bbbr$ and by extension. Then for each $z=\pi_V(x) \in \pve$,
$$ \int \phi_r^m(z-w)d\mu_V (w) \ \geq\  \frac{1}{C^{m}_{r }(E)}.$$ 
By Lemma \ref{potbound} 
$$N_r(\pve)\ \leq\   c_{m,m} \lceil\log ({\rm diam}(\pve) / r)+1\rceil C^{m}_{r}(E),$$
so 
$$\frac{\log  N_r(\pve)}{-\log r}
\   \leq\  \frac{\log\big(c_{m,m} \lceil\log ({\rm diam}(\pve) / r)+1\rceil\big)}{-\log r}
\ + \ \frac{\log C^{m}_{r}(E)}{-\ \log r}.$$
Taking lower and upper  limits as $r\searrow 0$, we conclude that $\lbd \pi_V E   \leq    \lbd^m E$ and $\ubd \pi_V E   \leq    \ubd^m E$ for all $V\in G(n,m)$.
\hfill$\Box$
\medskip

Note that a similar argument shows that $ \lbd f(E) \  \leq \   \lbd^m E$ for every Lipschitz map $f: E \to \mathbb{R}^m$.

Almost sure equality in Theorem \ref{mainA}(ii) and (iii) is more or less a particular case of the corresponding parts of Theorem \ref{mainB} so we combine the proofs. We first need a lemma to estimate the measure of the subspaces $V$ onto  which the projection of two given points are close to each other. We assume that the Grassmanian $G(n,m)$ is equipped with some natural locally $m(n-m)$-dimensional metric $d$, and ${\mathcal H}^t$ denotes $t$-dimensional Hausdorff measure on $G(n,m)$, defined with respect to this metric.

\begin{lem}\label{compl}
(a) There is a number $a_{n,m}>0$ depending only on $n$ and $m$ such that 
\be
\phi_r^m (x-y)\ \leq \ \gamma_{n,m}\big\{V: |\pi_Vx-\pi_Vy| \leq r\big\}\ \leq\ a_{n,m}\,\phi_r^m (x-y)
\qquad (x,y \in \rn, r>0).\label{bound}
\ee

(b) Let $0<s\leq m$ and let $K \subset  G(n,m)$ be a Borel set with Hausdorff measure ${\mathcal H}^{m(n-m) - (m-s)} (K)  >  0$. Then there is a Borel measure $\tau$ supported by $K$ with $\tau(K)>0$ and a number  $a_K>0$ such that 
\be
\tau\big\{V: |\pi_Vx-\pi_Vy| \leq r\big\}\ \leq\ a_K\,\phi_r^s (x-y)
\qquad (x,y \in \rn, r>0).\label{boundnu}
\ee

\end{lem}
\begin{proof}
(a) Note that  $\phi^m_r(x)$ is comparable to the proportion of the subspaces $V\in G(n,m)$ for which the  $r$-neighbourhoods of  the orthogonal subspaces to $V$ contain $x$, specifically, for all $1\leq m < n$ there are numbers $a_{n,m}>0$ such that 
$$
\phi_r^m (x) \ \leq\ \gamma_{n,m}\big\{V: |\pi_Vx| \leq r\big\}\ \leq\ a_{n,m}\,\phi_r^m (x)
\qquad (x\in \rn, r>0).
$$
This standard geometrical estimate can be obtained in many ways, see for example \cite[Lemma 3.11]{Mat}. One approach is to normalise to the case where $|x|$ = 1 and then estimate the (normalised) $(n-1)$-dimensional spherical area of $S \cap \{y: {\rm dist}(y,V^\perp)\leq r\}$, that is  the intersection of the unit sphere $S$ in $\rn$ with the `tube' or `slab' of points within distance $r$ of some $(n-m)$-dimensional subspace $V^\perp$ of $\rn$. Linearity then gives \eqref{bound}.

(b) By Frostman's Lemma, see \cite{Mat,Mat2}, there is a Borel probability measure $\tau$ supported on a compact subset of $K$ and $a>0$ such that 
\be\label{frost}
\tau (B_G(V,\rho)) \leq a\rho^{m(n-m) - (m-s)}\qquad (x\in \rn, \rho >0),
\ee
where $B_G(V,\rho)$ denotes the ball in $G(n,m)$ of centre $V$ and radius $\rho$ with respect to the metric $d$.
This ensures that the subspaces in $K$ cannot be too densely concentrated, and a geometrical argument gives
\be
 \tau\big\{V: |\pi_Vx| \leq r\big\}\ \leq\ a_K\,\phi_r^s (x)
\qquad (x\in \rn, r>0)
\ee
for some $a_K>0$, see \cite{Mat4} or \cite[(5.12)]{Mat2} for more details. 
\end{proof}

\noindent{\it Proof of Theorem \ref{mainA} (a.s. equality) and Theorem \ref{mainB}, parts (ii) and (iii)}. 
As before we may take $E$ to be compact. Let
$$K = \{V \in  G(n,m) : \ubd \pi_V E \,  < \,   \ubd^s E\};$$
then $K$ is a Borel set.
Suppose, for a contradiction, that ${\mathcal H}^{m(n-m) - (m-s)} (K)  >  0$.   By Lemma \ref{compl}(b) there is a measure $\tau$ supported by $K$ with $\tau(K)>0$ and satisfying \eqref{boundnu}.
For  $\mu  \in {\mathcal M}(E)$ and $V\in G(n,m)$, write $\mu_V$ for the projection of $\mu$ onto $V$ defined by 
$\int f(w) d\mu_V (w) = \int f(\pi_V(x)) d\mu(x)$ for continuous $f$ on $V$ and by extension. Using Fubini's theorem, 
\begin{align}
\int  (\mu_V \times \mu_V)&\big\{(w, z)\in V\times V:  |w-z| \leq r\big\}d\tau(V) \nonumber\\
&=\ \int (\mu \times \mu)\big\{(x, y):  |\pi_V x -\pi_V y| \leq r\big\}d\tau(V)\nonumber\\
&=\ \int \int\tau\big\{V\in G(n,m):  |\pi_V x -\pi_V y| \leq r\big\}d\mu(x)d\mu(y)\nonumber\\
&\leq \ a_K \int\int \phi_r^s (x-y)d\mu(x)d\mu(y)\label{pipi}
\end{align}
by \eqref{boundnu}.
If  $ \ubd^s E >t'>t>0 $ then $C_{r_k}^s(E) \geq {r_k}^{-t'}$ for a sequence $r_k \searrow 0$, where we may assume that $0< r_k \leq 2^{-k}$ for all $k$. Thus for each $k$ there is a measure $\mu^k  \in {\mathcal M}(E)$ such that 
$$\int\int \phi_{r_k}^s(x-y)d\mu^k(x)d\mu^k(y)\leq {r_k}^{ t'}.$$
Applying \eqref{pipi} to  each $\mu^k$ and summing over $k$,
\begin{align*}
\int\Big(\sum_{k=1}^\infty r_k^{- t}  & (\mu_V^k \times \mu_V^k)\big\{(w, z)\in V\times V:  |w-z| \leq r_k\big\}\Big)d\tau(V)\\
   &\leq\ a_K\sum_{k=1}^\infty r_k^{- t}\int\int \phi_{r_k}^s (x-y)d\mu^k(x)d\mu^k(y)\\
  &\leq\ a_K\sum_{k=1}^\infty r_k^{ (t'-t)} \ \leq \ a_K\sum_{k=1}^\infty 2^{-k(t'-t)}\ <\  \infty.
\end{align*}
Thus, for $\tau$-almost all $V$  there is a number $M_V< \infty$ such that 
$$ (\mu_V^k \times \mu_V^k)\big\{(w, z)\in V\times V:  |w-z| \leq r_k\big\}\leq M_V r_k^{ t}$$
for all $k$.
For such $V$, Lemma \ref{genbound} implies that 
$$N_{r_k} (\pve )\ \geq\  c_m M_V^{-1}  r_k^{- t}$$
 for all $k$, as the projected measures $\mu_V^k$ are supported by $\pve\subset V$. 
Hence $\varlimsup_{r\to 0} \log N_{r}(\pve)/-\log r \geq  t$. This is so for all $t< \ubd^s E$, so  
$\ubd \pve \  \geq\    \ubd^s E$ for $\tau$-almost all $V\in G(n,m)$, contradicting that $\tau (K) >0$.

The inequality for the lower dimensions for almost all $V$ follows in a similar manner, noting that it is enough to take  $r =2^{-k}, k\in \mathbb{N}$ when considering the limits as $r\searrow 0$ in the definitions of lower box dimension and lower box dimension profiles.

Thus we have proved Theorem \ref{mainB}(ii) and (iii).  Almost sure equality in Theorem \ref{mainA}(ii) and (iii) follows in exactly the same way by taking $s=m$, replacing  $\tau$ by the restriction of  $\gamma_{n,m}$ to $K$ and using \eqref{bound} at \eqref{pipi} to get a similar contradiction if $\gamma_{n,m}(K)>0$. 
\hfill$\Box$
\medskip

We remark that in the above proof we reached a contradiction to the fact that ${\mathcal H}^{m(n-m) - (m-s)} (K)  >  0$. Thus a slightly stronger conclusion in terms of measures is valid, namely that
$${\mathcal H}^{m(n-m) - (m-s)} \{V \in  G(n,m) : \ubd \pi_V E \,  < \,   \ubd^s E\}\ =\ 0.$$

\subsection{Proof of Theorem \ref{mainC} parts (ii) and (iii)}
\label{fouriercase}

Theorem \ref{mainC} gives upper bounds for the size of the exceptional directions for box dimensions of projections in terms of 
$\lbd^s E$ or $\ubd^s$ when $s\geq m$. We use a Fourier transform approach, analogously to the Hausdorff dimension case stated in Theorem \ref{mainC}(i), see \cite{Fa1}. 

We define the {\it Fourier transform}  of a function $f\in L^1(\bbbr^n)$ and a finite measure $\mu$ on $\bbbr^n$ by
$${\widehat f}(\xi)= \int f(x) {\rm e}^{{\rm i}\,\xi\cdot x} dx,\qquad 
{\widehat \mu}(\xi)= \int {\rm e}^{ {\rm i}\,\xi\cdot x} d\mu(x)\qquad (\xi \in \rn),$$
with the definitions extending to distributions in the usual way.

Fourier transforms of radially symmetric functions can be expressed as integrals against Bessel functions, see \cite[Section 3.3]{Mat2}, and in particular, for $s>n, r>0$, the kernels $\phi_r^s$ on $\bbbr^n$ transform as distributions to
$$\widehat{ \phi_r^s}(\xi)\ = \ c_n s |\xi|^{-n-1+s}\, r^s\int_{r|\xi|}^\infty J_{n/2}(u)\, u^{n/2 - s-1} du
\qquad (\xi \in \bbbr^n, r>0),$$
where $J_{n/2}$ is the Bessel function of order $n/2$ and $c_n$ depends only on $n$ (this form  follows from integrating the usual radial transform expression by parts). However, this oscillating transform is difficult to work with, so we introduce an alternative kernel $\psi_r^s$ that is equivalent to $\phi_r^s$ and which has strictly positive Fourier transform. Thus for $0<s<n$ and $ r>0$ we define $ \psi_r^s:\bbbr^n \to \bbbr^+$ by the convolution
\be\label{kerpsi}
 \psi_r^s(x)\  :=\ \big( |\cdot |^{-s} \ast {\rm e}\big)\Big(\frac{x}{r}\Big)\ = \int  |y |^{-s}{\rm e}\Big(\frac{x}{r}-y\Big)dy
\ee
where for convenience we  write 
$$
 {\rm e}(x)\ :=\ \exp\Big(-\frac{1}{2}|x|^2\Big)\quad\qquad (x\in \bbbr^n),\label{edef}
$$
 and also 
$$
 {\rm e}_r(x)\ :=\ {\rm e}\Big(\frac{x}{r}\Big) =\ \exp\Big(-\frac{1}{2}\Big|\frac{x}{r}\Big|^2\Big)\quad\qquad (x\in \bbbr^n, r>0).$$
In particular
\be\label{r1}
 \psi_r^s(x)\ \equiv\  \psi_1^s\Big(\frac{x}{r}\Big).
\ee
The following lemma summarises the key properties of $ \psi_r^s$.

\begin{lem}\label{psiprops}
For $0<s<n$  let $ \psi_r^s$ be  as in \eqref{kerpsi}. Then

{\rm (a)} 
there are constants $c_1,c_2 > 0$ depending only on $n$ and $s$  such that 
\be\label{comp}
c_1 \psi_r^s(x)\ \leq \ \phi_r^s(x)\ \leq\  c_2 \psi_r^{s}(x)
\qquad (x\in \bbbr^n, r>0);
\ee

{\rm (b)}
there is a constant $c_3 $ depending only on $n$ and $s$  such that 
\be\label{psitran}
\widehat{ \psi_r^s}(\xi)\ =\ c_3 r^s|\xi|^{s-n} {\rm e}(r\xi)\qquad (\xi\in \bbbr^n, r>0).
\ee
\end{lem}
\begin{proof}
{\rm (a)}  By \eqref{r1}  it is enough to establish \eqref{comp} when $r=1$.  
Then
\begin{eqnarray}
 \psi_1^s(x)   &= &\int  |y |^{-s}{\rm e}(x-y)dy \label{intcomp1}\\
   &\geq & c\int  |y |^{-s}\chi_{B(0,1)}(x-y)dy \ =: \ c\, J(x) \nonumber
\qquad (x\in \bbbr^n),
\end{eqnarray}
where $c= \exp(-\frac{1}{2})$ and  $\chi_{B(0,1)}$ is the indicator function of the unit ball. By obvious estimates, writing  $v_n$ for the volume of $B(0,1)$, if $|x| \leq 1$ then $J(x) \geq 2^{-s}v_n$ and if $|x| > 1$ then $J(x) \geq (2|x|)^{-s}v_n$. The right-hand inequality of \eqref{comp} follows for some $c_2>0$ when $r=1$ and thus for all $r>0$.

For the left-hand inequality, fixing $M>n$, there
is a constant $c>0$ such that ${\rm e}(x) \leq c\big(1+ |x|\big)^{-M}$ for all $ x\in \bbbr^n$, so from \eqref{intcomp1},
\be
 \psi_1^s(x)  \ \leq \  c \int_{\bbbr^n}  |y |^{-s}\big(1+ |x-y|\big)^{-M} dy\label{intcomp}
\qquad (x\in \bbbr^n).
\ee
Splitting the domain of integration of \eqref{intcomp} into regions $|y|\leq 1$ and $|y|> 1$ easily shows that the integral is bounded. Then splitting the domain into regions $|y|\leq \frac{1}{2}|x|$ and $|y|> \frac{1}{2}|x|$ gives upper bounds of orders $O(|x|^{n-s-M})$ and $O(|x|^{-s})$ respectively, so a bound of $O(|x|^{-s})$ overall. Thus the left-hand inequality of \eqref{comp} follows for a suitable $c_1$ when  $r=1$ and so for all $r>0$.

\medskip

{\rm (b)} Note that $\widehat{ |\cdot |^{-s}} =c|\cdot |^{s-n}$ in the distributional sense, where $c $ depends only on $n$ and $s$, see \cite[Theorem 3.6]{Mat2}, and  that ${\widehat {\rm e}}(\xi)\ = (2\pi)^{n/2}{\rm e}(\xi)$. Using the convolution theorem we would hope that
$$\widehat{\psi_1^s}(\xi)\ =\ \big(\widehat{|\cdot |^{-s} \ast  {\rm e}}\big)(\xi)\ 
=\   \widehat{ |\cdot |^{-s}}(\xi)\, \widehat{ {\rm e}}(\xi)\ = \ (2\pi)^{n/2}c\, |\xi |^{s-n} {\rm e}(\xi); 
$$
 the validity of this is justified in \cite[Lemma 3.9]{Mat2}. By scaling, the Fourier transform of $\psi_r^s$ for all $r>0$ is given by \eqref{psitran}, where $c_3 = (2\pi)^{n/2}c$.
\end{proof}

We now express energies with respect to the kernel $\psi_r^s$ in terms of Fourier transforms.

\begin{prop}
Let $0<s<n$  and $\mu  \in {\mathcal M}(E)$. Then there are constants $c_{4}, c_5$ and $c_6 $  depending only on $n$ and $s$ such that 

\begin{eqnarray}
\int\int \psi_r^s(x-y)d\mu(x)d\mu(y)    
&= &c_{4} \int \widehat{\psi_r^s}(\xi)|\widehat{\mu}(\xi)|^2d\xi\label{ften1}\\
&= &c_{5} r^s  \int |\xi|^{s-n}{\rm e}(r\xi)|\widehat{\mu}(\xi)|^2 d\xi \label{ften2}
\end{eqnarray}
and
\be
\int\int {\rm e}_r(x-y)d\mu(x)d\mu(y)\  =\  c_6\, r^n  \int {\rm e}(r\xi)|\widehat{\mu}(\xi)|^2 d\xi. \label{ften3}
\ee
\end{prop}
\begin{proof}
Intuitively \eqref{ften1} follows by applying Parseval's formula and the convolution formula. Justification of this requires some care, by first working with approximations to $\mu$ given by $\mu\ast \delta_\epsilon$ where $\{\delta_\epsilon\}_{\epsilon>0}$ is an approximate identity. However, the proof follows exactly that for the Riesz kernel  $|\cdot|^{-s}$ given, for example,  in \cite[Theorem 3.10]{Mat2}. Equation \eqref{ften2} then follows from \eqref{psitran}. The identity  \eqref{ften3} follows in a similar way.
\end{proof} 

For each $V\in G(n,m)$ we may decompose $x\in \bbbr^n$  as $x= x_V+x_{V^\perp}$, where $x_V\in V$ and $x_{V^\perp}\in V^\perp$, and where appropriate we will write $x$ as $(x_V,x_{V^\perp})$ in the obvious way. Given $\mu  \in {\mathcal M}(E)$ we define Radon measures $\nu_V$ on each $V\in G(n,m)$ by
\be\label{nudef}
\int_{\rn} f(x_V) d\nu_V(x_V)\ =\  \int_V f(x_V) {\rm e}(x_{V^\perp}) d\mu(x_V,x_{V^\perp}) 
\ee 
for all continuous $f$ on $V$ and by extension.  Then $\nu_V$ is a weighted projection of $\mu$ onto $V$ and the support of $\nu_V$ is the projection of the support of $\mu$ onto $V$, in particular ${\rm spt}(\nu_V) \subset \pi_V E$. The Fourier transform of $\nu_V$ on $V$ is given for $\xi_V \in V$ by

\begin{eqnarray}
\widehat{\nu_V}(\xi_V)\ 
&=& \int_{\bbbr^n} \exp({\rm i}\,x_V\!\cdot\!\xi_V)  {\rm e}(x_{V^\perp})d\mu(x_V,x_{V^\perp})\nonumber\\
&=&(2\pi)^{-(n-m)/2}\int_{\bbbr^n}  \int_{V^\perp}\exp({\rm i}\, x_V\!\cdot\!\xi_V) \exp({\rm i}\, x_{V^\perp}\!\cdot\!\xi_{V^\perp}) {\rm e}(\xi_{V^\perp}) d\xi_{V^\perp}d\mu(x_V,x_{V^\perp})\nonumber\\
&=& (2\pi)^{-(n-m)/2} \int_{V^\perp} \widehat{\mu}(\xi_V,\xi_{V^\perp}){\rm e}(\xi_{V^\perp})d\xi_{V^\perp}\label{nuvtran},
\end{eqnarray}
using the transform of the symmetric ${\rm e}(\xi_{V^\perp})$ and Fubini's theorem.

Next we relate the transforms of the $\nu_V$ to that of $\mu$ for each $V$.

\begin{lem}\label{est1}
Let $0<\gamma<n$, let $\mu  \in {\mathcal M}(E)$ and let $\nu_V$ be defined by \eqref{nudef}. Then there is a constant $c_7$ depending only on $n$ and $m$ such that for all $V\in G(n,m)$ and $0<r<\frac{1}{2}$,
\be\label{vineq}
\int_V  |\widehat{\nu_V}(\xi_V)|^2 {\rm e}(r\xi_{V})d\xi_V
\ \leq\ c_7   \int_{\bbbr^n} |\widehat{\mu}(\xi) |^2 {\rm e}(r\xi)\exp(\textstyle{-\frac{1}{4}}|\xi_{V^\perp}|^2) d\xi.
\ee
\end{lem}
\begin{proof}
Applying Schwarz's inequality to \eqref{nuvtran}, for some $c_7>0$,
$$ |\widehat{\nu_V}(\xi_V)|^2 \ \leq \ c_7 \int |\widehat{\mu}(\xi_V,\xi_{V^\perp})|^2 {\rm e}(\xi_{V^\perp})d\xi_{V^\perp}. $$
Thus   for all $V\in G(n,m)$,  $\xi = (\xi_V,\xi_{V^\perp})\in \bbbr^n$ and $0<r<\frac{1}{2}$,
\begin{align*}
|\widehat{\nu_V}(\xi_V)|^2  {\rm e}(r\xi_{V})\
&\leq \ c_7 \int |\widehat{\mu}(\xi_V,\xi_{V^\perp})|^2 {\rm e}(r\xi_{V}){\rm e}(\xi_{V^\perp})d\xi_{V^\perp}\\
&= \ c_7 \int |\widehat{\mu}(\xi_V,\xi_{V^\perp})|^2  \exp(\textstyle{-\frac{1}{2}r^2|\xi_V|^2  -\frac{1}{2}}|\xi_{V^\perp}|^2 )d\xi_{V^\perp}\\
&\leq\ c_7 \int |\widehat{\mu}(\xi_V,\xi_{V^\perp})|^2  \exp\big(\textstyle{-\frac{1}{2}}r^2(|\xi_V|^2  +|\xi_{V^\perp}|^2 )  \textstyle{-\frac{1}{4}}|\xi_{V^\perp}|^2   \big)d\xi_{V^\perp}\\
&= \ c_7 \int |\widehat{\mu}(\xi_V,\xi_{V^\perp})|^2 {\rm e}(r\xi) \exp(\textstyle{-\frac{1}{4}}|\xi_{V^\perp}|^2)d\xi_{V^\perp}.
\end{align*}
Integrating with respect to  $\xi_V$ gives \eqref{vineq}.
\end{proof}

To enable us to integrate \eqref{vineq} over $V$ we need a strightforward bound for the integral of $\exp(\textstyle{-\frac{1}{4}}|\xi_{V^\perp}|^2)$.

\begin{lem}\label{angint}
Let $W$ be an analytic subset of $G(n,m)$ with ${\mathcal H}^t (W) >0$ where $0\leq (m-1)(n-m)<t< m(n-m)$. Then there exists a Borel probability measure $\tau$ supported by $W$ and a constant $c_8$ depending only on $n,m$ and $t$ such that for all $\xi \in \bbbr^n$
$$\int \exp(\textstyle{-\frac{1}{4}}|\xi_{V^\perp}|^2)d\tau(V)
\ \leq\ c_8\, |\xi|^{{(m-1)(n-m)-t}}.
$$
\end{lem}

\begin{proof}
A consequence of \eqref{frost} is that there exists a probability measure $\tau$ supported by $W$ and $c>0$ such that 
$$\tau\big(V\in G(n,m): |\xi_{V^\perp}| \leq \lambda\big)\ \leq\ c\bigg(\frac{\lambda}{|\xi|}\bigg)^{t-(m-1)(n-m)}\quad (\lambda>0, |\xi|>0),$$
see \cite[(5.11)]{Mat2}. Then 
\begin{eqnarray*}
\int \exp(\textstyle{-\frac{1}{4}}|\xi_{V^\perp}|^2)d\tau(V)
&=&
\int_0^\infty \textstyle{\frac{1}{2}}\lambda  \exp(\textstyle{-\frac{1}{4}}\lambda^2) \tau\big(V: |\xi_{V^\perp}| \leq \lambda\big) d\lambda\\
&\leq&
{\textstyle\frac{1}{2}}c\int_0^\infty \lambda \exp(-{\textstyle\frac{1}{4}}\lambda^2) \bigg(\frac{\lambda}{|\xi|}\bigg)^{t-(m-1)(n-m)} d\lambda\\
&=&c_8\, |\xi|^{(m-1)(n-m)-t}\qquad (\xi \in\rn),
\end{eqnarray*}
since the integral with respect to $\lambda$ is finite.
\end{proof}

\noindent{\it Proof of Theorem 1.3. }
Let $0<d<d'<\lbd^{m+\gamma} E$. Then for each   $0<r\leq\frac{1}{2}$ there is a measure $\mu^r \in {\mathcal M}(E)$ such that the energy
\be\label{energmg}
r^{m+\gamma}  \int |\xi|^{m+\gamma-n}{\rm e}(r\xi)|\widehat{\mu^r}(\xi)|^2 d\xi \ = \  c_5^{-1}\int\int \psi_r^{m+\gamma}(x-y)d\mu^r(x)d\mu^r(y)\  \ \leq\  c\,r^{d'},
 \ee
   where $c$ is independent of $r$, using \eqref{ften2}, \eqref{comp} and \eqref{dimpro}. 
Let $\nu_V^r$ be the weighted projection of $\mu^r$ onto $V$ derived from $\mu^r$ as in \eqref{nudef}. Let $W_r \subseteq G(n,m)$ be the Borel set
\be
W_r  
\ :=\ \Big\{V \in  G(n,m): r^m \int_V |\widehat{\nu_V^r}(\xi_V) |^2 {\rm e}(r\xi_{V})d\xi_V \geq r^{d-\gamma} \Big\}.\label{wr}
\ee
Let $W = \limsup_{k\to\infty}W_{2^{-k}}$.
Let $t = m(n-m)-\gamma$ so $(m-1)(n-m)-t =\gamma +m-n$. Suppose, for a contradiction, that ${\mathcal H}^t (W) >0$. By Lemma \ref{angint} there is a probability measure $\tau$ supported by $W$ satisfying 
\be\label{roteq1}
 \int \exp(\textstyle{-\frac{1}{4}}|\xi_{V^\perp}|^2)d\tau(V)\ \leq\ c_8\, |\xi|^{\gamma +m-n}.
 \ee
For convenience write $\mu^k = \mu^{2^{-k}}$ and $\nu_V^k=\nu_V^{2^{-k}}$. Then, using \eqref{vineq}, Lemma \ref{est1},  \eqref{roteq1} and \eqref{energmg},
\begin{eqnarray*}
\sum_{k=1}^\infty \tau(W_{2^{-k}})
&=& 
\sum_{k=1}^\infty \tau\Big\{V :  \int_V |\widehat{\nu_V^{k}}(\xi_V) |^2 {\rm e}(2^{-k}\xi_{V})d\xi_V \geq 2^{-k(d-m-\gamma)} \Big\} \\
&\leq& 
\sum_{k=1}^\infty \tau\Big\{V :  c_7   \int_{\bbbr^n} |\widehat{\mu^{k}}(\xi) |^2 {\rm e}(2^{-k}\xi)\exp(\textstyle{-\frac{1}{4}}|\xi_{V^\perp}|^2) d\xi \geq 2^{-k(d-m-\gamma)}  \Big\} \\
&\leq& 
\sum_{k=1}^\infty c_7\,2^{-k(\gamma+m-d)} \int_V  \int_{\bbbr^n} |\widehat{\mu^{k}}(\xi) |^2 {\rm e}(2^{-k}\xi)\exp(\textstyle{-\frac{1}{4}}|\xi_{V^\perp}|^2) d\xi d\tau (V)\\
&\leq& c_7c_8\,2^{-k(\gamma+m-d)} \int_{\bbbr^n} |\widehat{\mu^{k}}(\xi) |^2 {\rm e}(2^{-k}\xi) |\xi|^{\gamma+m-n} d\xi \\
&\leq& 
\sum_{k=1}^\infty c_6c_7c_8\,2^{-k(d'-d)}<\infty. 
\end{eqnarray*}
By the Borel-Cantelli lemma ${\mathcal H}^t (W) =0$, a contradiction, so $\hdd W\leq t$.

For all $V\notin W$, by \eqref{ften3},
\begin{eqnarray*}
(\nu_V^k \times \nu_V^k)\big\{(w,z) \in V\times V : |w-z| \leq 2^{-k}\big\}
&\leq & {\rm e}^{1/2} \int\int {\rm e}_{2^{-k}}(w-z)d\nu_V^k(w)d\nu_V^k(z)\\
&= & {\rm e}^{1/2}c_6\, 2^{-km} \int_V |\widehat{\nu_V^k}(\xi_V) |^2 {\rm e}(2^{-k}\xi_{V})d\xi_V\\
&\leq & {\rm e}^{1/2}c_6\, 2^{-k(d-\gamma)}
\end{eqnarray*}
 for all sufficiently large $k$, by \eqref{wr}.
 Since $\nu_V^{k}$ is supported by $\pi_V E$,  Lemma \ref{genbound} implies that  there is $c'>0$ such that $N_{2^{-k}}(\pi_V E) \geq c'\, 2^{k(d-\gamma)}$ for all sufficiently large $k$, so $\lbd(\pi_V E)\geq d-\gamma$, since when finding box-dimensions it is enough to consider a sequence of scales  $r=2^{-k}\, (k\in\mathbb{N})$. This is true for all $0<d<d'<\lbd^{m+\gamma} E$, so 
$\lbd(\pi_V E)\geq\lbd^{m+\gamma} E -\gamma$ for all $V\in G(n,m)$ except for a set of Hausdorff dimension at most $t$, giving Theorem \ref{mainC}(ii).

The proof of Theorem \ref{mainC}(ii) is similar, taking $0<d<d'<\ubd^{m+\gamma} E$ and summing over those $r=2^{-k}$ for which \eqref{energmg} is satisfied. 
\hfill$\Box$

\section{Packing dimensions}\label{Secpac}
\setcounter{equation}{0}
\setcounter{theo}{0}

In this section we show how the results for box-counting dimensions carry over to the packing dimensions.

Packing measures and dimensions were introduced by Taylor and Tricot \cite{Tri,TT} as a type of dual to Hausdorff measures and dimensions, see \cite{Fa,Mat} for more recent expositions. Whilst, analogously to Hausdorff dimensions, packing dimensions can be defined by first setting up packing measures,  an equivalent definition in terms of  upper box dimensions  of countable coverings of a set is often more convenient in practice. Thus for $E \subset \mathbb{R}^n$ we may define the {\it packing dimension} of $E$ by
\be\label{packdef}
\pkd E \ = \  \inf \Big\{\sup_{1\leq i<\infty} \ubd E_i : E \subset \bigcup_{i =1}^\infty E_i\Big\};
\ee
since the box dimension of a set equals that of its closure, we can assume that the sets $E_i$ in \eqref{packdef} are all compact.

It is natural to make an analogous definition of the {\it packing dimension profile} of $E \subset \mathbb{R}^n$ for $s>0$ by
\be\label{packpro}
\pkd^s E \ = \  \inf \Big\{\sup_{1\leq i<\infty} \ubd^s E_i : E \subset \bigcup_{i =1}^\infty E_i \mbox{ with each } E_i  \mbox{ compact} \Big\}.
\ee

With this definition, properties of packing dimension can be deduced from corresponding properties of upper box dimension. Thus we get an immediate analogue of Corollary \ref{sgeqn}.
\begin{cor}
Let $E\subset \mathbb{R}^n$. If $s\geq n$ then 
$$\pkd^s E\   =\   \pkd E.$$
\end{cor}

With these definitions we can deduce the packing dimension parts (iv) of our main theorems from the corresponding upper box dimension parts (iii). For this we need the following `localisation' property.

\begin{prop}\label{goodsubset}
Let $E\subset \mathbb{R}^n$ be a Borel set such that $\pkd^s E>t$. Then there exists a non-empty compact $F\subset E$ such that $\pkd^s (F\cap U )>t $ for every open set $U$ such that $F\cap U \neq \emptyset$.
\end{prop}

\begin{proof}
In the special case where $E$ is compact there is a  short proof is based on \cite[Lemma 2.8.1]{BP}.
Let ${\mathcal B}$ be a countable basis of open sets that intersect $E$. Let 
$$ F\ =\ E\setminus \bigcup\big\{V \in {\mathcal B}: \pkd^s (E\cap V)\leq t\big\}.$$
Then $F$ is compact and, since  $ \pkd^s$ is countably stable, $ \pkd^s F >t$ and furthermore $ \pkd^s (E\setminus F) \leq t$.

Suppose for a contradiction that  $U$ is an open set such that $F\cap U \neq \emptyset$ and  $\pkd^s (F\cap U )\leq t $. As $ {\mathcal B}$ is a basis of open sets we may find  $V\subset U$ with $V \in {\mathcal B}$ such that $F\cap V \neq \emptyset$ and  $\pkd^s (F\cap V )\leq t$. Then
$$\pkd^s (E\cap V )\ \leq\ \max\{\pkd^s (E\setminus F ), \pkd^s (F\cap V )\} \ \leq t,$$
so $E\cap V$ is disjoint from $F$ by definition of  $F$, which contradicts that $F\cap V \neq \emptyset$.

For a general Borel set $E$ with $\pkd^s E>t$ we need to find a compact subset $E'\subset E$ with $\pkd^s E'>t$ which then has a suitable subset as above.
Whilst this is intuitively natural, I am not aware of a simple direct proof from the definition \eqref{packpro} of packing dimension profiles in terms of  box dimension profiles. However the existence of such a set $E'$ is proved in \cite{How} using packing-type measures. In that paper, measures ${\mathcal P}^{s,d}$ are constructed so that $\pkd^s E = \inf\{d: {\mathcal P}^{s,d}(E) <\infty\}$. If $\pkd^s E>t$ then ${\mathcal P}^{s,t}(E) = \infty$ and \cite[Theorem 22]{How} gives a construction of a compact  $E'\subset E$ with ${\mathcal P}^{s,t}(E') = \infty$, so that $\pkd^s E'>t$. The above argument can then be applied to $E'$.
\end{proof}

With the definitions of $\pkd$ and  $\pkd^s$ we can transfer the results on projections and exceptional sets from upper box dimensions to packing dimensions.
\medskip

\noindent{\it Proof of part (iv) of Theorems \ref{mainA}, \ref{mainB} and \ref{mainC}}. 
If $t> \pkd^{s} E$ we may cover $E$ by a countable collection of compact sets $E_i$ such that $\ubd^{s} E_i <t$.  By Theorem \ref{mainA}(iii), for all $V\in G(n,m)$,
$$\ubd \pi_V( E_i)\  \leq \  \ubd^{s} E_i \ \leq\  t, $$
for all $i$. Since  $\pi_V( E)\subset \bigcup_i  \pi_V( E_i)$,  $\pkd \pi_V( E)\leq t$ by \eqref{packdef}, so as this holds for all $t> \pkd^{s} E$, the inequality in Theorem \ref{mainA}(iv) follows.

We next derive Theorem \ref{mainB}(iv) from Theorem \ref{mainB}(iii). Let $0<s\leq m$ and let $t< \pkd^s E$. By Proposition \ref{goodsubset} we may find a  non-empty compact $F\subset E$ such that for every open $U$ that intersects $F$,  $\pkd^s (F\cap U )>t $,  so in particular $\ubd^s (F\cap {\overline U} )>t $. As  $\mathbb{R}^n$ is separable, there is a countable basis $\{U_i\}_{i=1}^\infty$ of open sets that intersect $F$.  
For each $i\in \mathbb{N}$ let
\be\label{wi}
W_i\ = \ \big\{V\in G(n,m): \ubd \pi_V(F\cap {\overline U}_i)\ <\  \ubd^{s}(F\cap {\overline U}_i)\big\}.
\ee
By Theorem \ref{mainB}(iii) $\hdd W_i  \leq m(n-m) - (m-s)$ for all $i$, so writing $W= \bigcup_{i=1}^\infty  W_i$, it follows that $\hdd W  \leq m(n-m) - (m-s)$.

Let $V\notin W$. If $\{K_j\}_{j=1}^\infty$ is any cover of the compact set $\pi_V(F)$ by a countable collection of compact sets,  Baire's category theorem implies that there is an index  $k$ and an open set $U$ such that $\emptyset \neq \pi_V(F)\cap U\subset \pi_V(F)\cap K_k$. There is some $U_i$ such that $\pi_V(F \cap U_i)\subset \pi_V(F)\cap U$, so in particular 
$$\ubd (\pi_V(F)\cap K_k)\  \geq\ \ubd (\pi_V(F)\cap {\overline U})\  \geq \ \ubd \pi_V(F \cap {\overline U}_i)\  \geq \   \ubd^{s}(F\cap {\overline U}_i)\ >\ t$$
as $V\notin W_i$. Thus $\pkd \pi_V E \geq \pkd \pi_V F\geq t$ if $V\notin W$. This is true for all $t< \pkd^s E$, so the conclusion follows from taking a countable sequence of $t$ increasing to $\pkd^s E$.

The derivation of part (iv) of Theorem \ref{mainC} from part (iii) is virtually identical, except at \eqref{wi} we take
$$ W_i\ = \ \big\{V\in G(n,m): \ubd \pi_V(F\cap {\overline U}_i)\ <\  \ubd^{m+\gamma}(F\cap {\overline U}_i)-\gamma\big\},$$
and note that $\hdd W_i  \leq m(n-m) - \gamma$ for each $i$.

Finally, $\gamma_{n,m}$-almost sure equality in Theorem \ref{mainA}(iv) again follows from part (iii) by the same argument, this time taking $W_i$ as in \eqref{wi} with $s=m$ and noting that $\gamma_{n,m} W_i=0$ so that $\gamma_{n,m} W=0$.
\hfill $\Box$
\medskip

\section{Inequalities}\label{ineqs}
\setcounter{equation}{0}
\setcounter{theo}{0}

A number of inequalities are satisfied by the dimension profiles; these were obtained for packing dimension profiles in\cite[Section 6]{FH2} but their derivation is more direct using our capacity approach. In particular inequality \eqref{ineq15} may be written in three equivalent ways which give different insights into the behaviour of the profiles.
\begin{prop}
Let $E\subset \mathbb{R}^n$ and let $d(s)$ denote any one of  $\lbd^s E, \ubd^s E$ or $\pkd^s E$. Then  for  $0< s\leq t$,
\be\label{ineq1}
0\ \leq\ d(s) \ \leq\ d(t)\  \leq \ n, 
\ee
and
\be\label{ineq15}
 \frac{d(t)}{1 +(1/s -1/t)d(t)}\ \leq\ d(s) \ \leq s.
\ee
If $d(s)>0$ then \eqref{ineq15} is equivalent  to
\be\label{ineq2}
0\ \leq \ \frac{1}{d(s)} - \frac{1}{s} \ \leq\  \frac{1}{d(t)} - \frac{1}{t},
\ee
giving the Lipschitz form
\be\label{ineq3}
d(t)-d(s) \ \leq\  \frac{d(s)d(t)}{st}(t-s)\ \leq\ t-s.
\ee
\end{prop}

\begin{proof}
First note that it is enough to prove \eqref{ineq1} and \eqref{ineq15} for $d(s)=\lbd^s E$ and $d(s)=\ubd^s E$. The analogues for  $d(s)=\pkd^s E$ then follow using the definition \eqref{packdef} of packing dimension profiles in terms of upper box dimension profiles. Note also that  \eqref{ineq2} and  \eqref{ineq3} come from simple rearrangements of  \eqref{ineq15}.

Inequality \eqref{ineq1} is immediate from the definitions since from  \eqref{ker} $\phi_r^s(x) \geq \phi_r^t(x)$ if $s\leq t$.
For the right-hand side of \eqref{ineq15} note that $C_r^s(E)^{-1} = \int\phi_r^s(x-y)d\mu_0(y) \geq r^s\int_{|x-y|\geq r}|x-y|^{-s}d\mu_0(y)$ for some $x\in E$, where  $\mu_0$ is an energy-minimising measure on $E$, and this last integral is bounded away from 0 for small $r$; taking lower or upper limits as $r\searrow 0$ gives the conclusion for box dimensions.

For the left-hand side of \eqref{ineq15}  let $0<r<R$, $0<s<t$ and $d>0$. Then for $\mu\in {\mathcal M}(E)$ and $x\in E$, splitting the integral and using H\"{o}lder's inequality,
\begin{eqnarray*}
\int \phi_r^s(x-y) d\mu(y) 
&\leq& \mu(B(x,R)) + \int_{|x-y|>R} \Big(\frac{r}{|x-y|}\Big)^sd\mu(y) \\
&=& \mu(B(x,R)) + r^sR^{-s}\int_{|x-y|>R} \Big(\frac{R}{|x-y|}\Big)^s d\mu(y) \\
&\leq & \mu(B(x,R)) +  r^sR^{-s}\bigg(\int_{|x-y|>R} \Big(\frac{R}{|x-y|}\Big)^t d\mu(y)\bigg)^{s/t} \\
&\leq & \int\phi_R^t(x-y) d\mu(y) +  r^sR^{-s}\bigg(\int\phi_R^t(x-y) d\mu(y)\bigg)^{s/t} \\
&\leq &R^{d} \bigg(R^{-d}\int\phi_R^t(x-y) d\mu(y)\bigg) +  r^sR^{s(d/t-1)}\bigg(R^{-d}\int\phi_R^t(x-y) d\mu(y)\bigg)^{s/t}
\end{eqnarray*}
Setting $R= r^{1/(1+(1/s -1/t)d)}$ this rearranges to 
$$r^{-d/(1+(1/s-1/t)d)} \int \phi_r^s(x-y) d\mu(y) \ \leq\ R^{-d}\int\phi_R^t(x-y) d\mu(y)
+\bigg(R^{-d}\int\phi_R^t(x-y) d\mu(y)\bigg)^{s/t}.$$

If $C_R^t (E) \geq R^{-d}$ for some $R$ then by Lemma \ref{equilib} there is a measure $\mu\in {\mathcal M}(E)$ such that the right-hand side of this inequality, and thus the left-hand side, is at most 2 for $\mu$-almost all $x$, so  
$C_r^s (E) \geq \frac{1}{2}r^{-d/(1+(1/s-1/t)d)}$ for the corresponding $r$. Letting $R\searrow 0$, it follows that
$$d(s)\ \geq\ \frac{d(t)}{1 +(1/s -1/t)d(t)},$$
where $d(\cdot)$ is either the lower or upper dimension profile, which rearranges to \eqref{ineq2}.
\end{proof}
\bigskip

Examples show that the inequalities \eqref{ineq2} give a complete characterisation of the dimension profiles that can be attained, see  \cite[Section 6]{FH2}. Setting $s=m$ and $t=n$ in inequalities \eqref{ineq15} gives \eqref{projin} along with similar inequalities for box dimensions, bounding the dimension profiles of $E$, and thus the typical dimensions of its projections, in terms of the dimension of $E$ itself.

\bibliographystyle{plain}

\bigskip
\end{document}